\documentclass[11pt]{article}
\usepackage{a4wide}
\usepackage{amsmath}
\usepackage{amsfonts}
\usepackage{amsthm}
\usepackage{amssymb}
\usepackage[pdftex]{hyperref}
\usepackage[all]{xy}
\usepackage{flafter}
\usepackage[utf8]{inputenc}
\usepackage{multirow}
\usepackage{enumitem}
\usepackage{color}

\usepackage{pgf,tikz,pgfplots}
\usetikzlibrary{plotmarks}
\usetikzlibrary{arrows}
 
\usetikzlibrary{shapes}
\usetikzlibrary{backgrounds}
\usetikzlibrary{external}
\usetikzlibrary{patterns}

\newcommand{\N}{{\mathbb N}}
\newcommand{\Z}{{\mathbb Z}}

\newcommand{\R}{{\mathbb R}}
\newcommand{\C}{{\mathbb C}}
\newcommand{\F}{{\mathbb F}}

\newcommand{\semi}{{\mathbb o}}
\DeclareMathOperator{\lie}{{\mathfrak g}}

\DeclareMathOperator{\Fix}{{\rm Fix}}
\DeclareMathOperator{\Aut}{{\rm Aut}}

\DeclareMathOperator{\Aff}{{\rm Aff}}

\DeclareMathOperator{\aff}{{\rm aff}}
\DeclareMathOperator{\grp}{{\rm grp}}
\DeclareMathOperator{\Endo}{Endo}

\DeclareMathOperator{\lcm}{{\rm lcm}}
\DeclareMathOperator{\Id}{Id}

\DeclareMathOperator{\GL}{{\rm GL}}
\newtheorem{theorem}{Theorem}[section]
\newtheorem{proposition}[theorem]{Proposition}
\newtheorem{lemma}[theorem]{Lemma}
\newtheorem{corollary}[theorem]{Corollary}
\newtheorem{definition}[theorem]{Definition}
\newtheorem{remark}{Remark}
\newtheorem{example}{Example}

\newcommand{\ra}{\rightarrow}

\makeatletter
\def\blfootnote{\xdef\@thefnmark{}\@footnotetext}
\makeatother
\begin{document}
\title{\bf Nielsen periodic point theory on infra-nilmanifolds}
\author{Gert-Jan Dugardein\\
KULeuven Kulak, E.\ Sabbelaan 53, B-8500 Kortrijk}
\date{\today}
\maketitle

\begin{abstract}
In this paper, we expand certain aspects of Nielsen periodic point theory from tori and nilmanifolds to infra-nilmanifolds. We show that infra-nilmanifolds are essentially reducible to the GCD and essentially toral. With these structural properties in mind, we develop a method to compute the full Nielsen-Jiang number $NF_n(f)$. We also determine for which maps it holds that $NF_n(f)=N(f^n)$, for all $n$.
\end{abstract}
%\blfootnote{\textit{Author's e-mail address:} gertjan.dugardein@kuleuven-kulak.be}
%\blfootnote{\textit{2010 Mathematics Subject Classification:} 37C25, 37C30, 55M20.}
%\blfootnote{\textit{Key Words:} Nielsen zeta function, Nielsen number, Lefschetz number, infra-nilmanifold.}
\section{Introduction}
In Nielsen theory, the purpose is to find a homotopy-invariant lower bound for the number of fixed points of a map. Similarly, Nielsen periodic point theory will try to find a homotopy-invariant lower bound for the number of periodic points. Both of these theories work for all continuous maps on connected, compact absolute neighborhood retracts, but in this paper we will mainly focus on manifolds.

\medskip

So, let $f:X\to X$ be a continuous map on a manifold. We call $x\in X$ a periodic point if $f^n(x)=x$ for a certain integer $n>0$. The smallest $n$ for which this holds, will be called the period of $n$. In Nielsen theory, we partition the fixed point set into fixed point classes and subsequently count the number of fixed point classes that can not disappear by using a homotopy. The resulting number is the Nielsen number $N(f)$. We can do a similar thing for $f^n$ and try to approximate the number of $n$-periodic points of $f$ by the number $N(f^n)$. However, this lower bound is not sharp in general. By studying the relations between fixed point classes for different iterates of $f$ more closely, it is possible to define a better lower bound, namely the full Nielsen-Jiang periodic number $NF_n(f)$. This number has already been studied extensively for maps on tori and nilmanifolds and in this paper we will extend some of these results to infra-nilmanifolds.

\medskip

In the first two sections, we will give an introduction to the theory of infra-nilmanifolds and to Nielsen periodic point theory. Subsequently, we will prove that infra-nilmanifolds are essentially reducible to the GCD (Theorem \ref{thmessredgcd}) and essentially toral (Theorem \ref{thmesstor}). These structural properties also hold for maps on nilmanifolds and tori and are a necessary first step in order to be able to compute $NF_n(f)$ on infra-nilmanifolds. We will also prove that for all $n$, $NF_n(f)=N(f^n)$ for a large class of maps on infra-nilmanifolds, namely the class of semi-hyperbolic maps (Theorem \ref{thmNF=Nf)}).

\medskip

Because maps on tori and nilmanifolds are weakly Jiang, the computation of $NF_n(f)$ turns out to be very doable on these manifolds. On infra-nilmanifolds, however, this is not the case. In the penultimate section, we will therefore develop a method that makes the computation of $NF_n(f)$ easier. Sometimes, though, the computation still can be quite hard, but this might be inherent to the problem, because, by applying our method to several examples, it becomes apparent that the expression for $NF_n(f)$ can be very complex.

\medskip

In the last section, we will look specifically at affine maps on infra-nilmanifolds. In general, these maps behave better than arbitrary continuous maps. In Theorem \ref{theorem uiteindelijk boost inessentieel}, we prove that, under very mild conditions, these affine maps can only be Wecken (which means that $\#\Fix(f)=N(f)$) at every level if and only if they are semi-hyperbolic. This allows us to determine exactly for which maps $NF_n(f)=N(f^n)$, for all $n$ (Corollary \ref{cor NF=Nf}). 
\section{Infra-nilmanifolds}
Let $G$ be a connected, simply connected, nilpotent Lie group. The group of affine transformations on $G$, $\Aff(G)= G\semi \Aut(G)$, admits a natural left action on $G$:
\[ \forall (g,\alpha)\in \Aff(G),\, \forall h \in G: \;\;^{(g,\alpha)}h= g \alpha(h).\]Define $p:\Aff(G)=G\semi \Aut(G) \to \Aut(G)$ as the natural projection onto the second factor of the semi-direct product. 

\begin{definition} A subgroup $\Gamma \subseteq \Aff(G)$ is called \textbf{almost-crystallographic} if and only if $p(\Gamma)$ is finite and $\Gamma\cap G$ is a uniform and discrete subgroup of $G$. The finite group $F=p(\Gamma)$ is called the holonomy group of $\Gamma$.
\end{definition}

With these properties, the natural action of such a group $\Gamma$ on $G$ becomes properly discontinuous and cocompact. Moreover, when $\Gamma$ is torsion-free, this action is free, which makes the resulting quotient space $\Gamma\backslash G$ a compact manifold. This idea leads to the following definition. 

\begin{definition}
A torsion-free almost-crystallographic group $\Gamma\subseteq \Aff(G) $ is called an \textbf{almost-Bieberbach group}, and the corresponding manifold $\Gamma\backslash G$ is called an \textbf{infra-nilmanifold} (modeled on $G$).  
\end{definition} 

When the holonomy group is trivial, $\Gamma$ can be considered to be a lattice in $G$ and the corresponding manifold $\Gamma\backslash G$ is a nilmanifold. When $G$ is abelian, i.e. $G$ is isomorphic to $\R^n$, $\Gamma$ will be called a Bieberbach group and $\Gamma\backslash G$ a compact flat manifold. When $G$ is abelian and the holonomy group of $\Gamma$ is trivial, then $\Gamma\backslash G$ is a torus. Hence, infra-nilmanifolds are a natural generalization of nilmanifolds and tori.

\medskip

Now, define the semigroup $\aff(G)=G\semi \Endo(G)$. Note that $\aff(G)$ acts on $G$ in a similar way as $\Aff(G)$:\[ (d,D): \; G \rightarrow G:\; h \mapsto d D(h).\]The elements of this semigroup will be called affine maps, since $\aff(G)$ is merely a generalization of the semigroup of affine maps $\aff(\R^n)$ to the nilpotent case. One of the main advantages of working with infra-nilmanifolds, is the fact that every continuous map lies in the same homotopy class as a map induced by such an affine map with similar properties. These maps are often easier to handle and are therefore ideal to use in proving several theorems. This strategy will be often used throughout this paper, for example in the last section of this paper.

\begin{theorem}[K.B.\ Lee \cite{lee95-2}]
\label{leemaps} Let $G$ be a connected and simply connected nilpotent Lie group and suppose that $\Gamma, \Gamma'\subseteq \Aff(G)$ are two almost-crystallographic groups modeled on $G$. 
Then for any homomorphism $\varphi: \Gamma\rightarrow \Gamma'$ there 
exists an element $  (d, D)\in \aff(G)$ such that 
\[ \forall \gamma \in \Gamma: \; \varphi(\gamma) (d,D) =  (d,D) \gamma.\] 
\end{theorem}

We can consider the equality $ \varphi(\gamma) (d,D) =  (d,D) \gamma$ in $\aff(G)$, since $\Aff(G)$ is a subgroup of $\aff(G)$. With this equality in mind, when $\Gamma$ and $\Gamma'$ are torsion-free, it is easy to see that the affine map $(d,D)$ induces a well-defined map between infra-nilmanifolds:\[\overline{(d,D)}: \Gamma \backslash G \rightarrow \Gamma' \backslash G: \; \Gamma h \rightarrow \Gamma' d D(h),\]
which exactly induces the morphism $\varphi$ on the level of the fundamental groups.

\medskip

On the other hand, if we choose an arbitrary map $f:\Gamma\backslash G\ra \Gamma'\backslash G$ between two infra-nilmanifolds and choose a lifting $\tilde{f}:G \to G$ of $f$, then there exists a morphism $\tilde{f}_\ast:\Gamma\to \Gamma'$ such that $\tilde{f}_\ast(\gamma) \circ \tilde{f} = \tilde{f}\circ \gamma$, for all $\gamma\in \Gamma$. By Theorem~\ref{leemaps}, an affine map $(d,D)\in \aff(G)$ exists which also satisfies $\tilde{f}_\ast(\gamma) \circ (d,D)= (d,D)\circ \gamma$ for all $\gamma\in \Gamma$. Therefore, the induced map $\overline{(d,D)} $ and $f$ are homotopic. Hence, whenever we are studying homotopy-invariant properties for maps on infra-nilmanifolds, we are free to replace an arbitrary map $f$ by its affine counterpart.

\medskip

The map $(d,D)$ will be called an \textbf{affine homotopy lift} of $f$, while we will denote the map $\overline{(d,D)}$ as an \textbf{affine map on an infra-nilmanifold}.

\medskip

It might be noteworthy to mention that $(d,D)$ is not unique in the sense that it depends on the choice of lifting $\tilde{f}:G\to G$. For example, from \cite{lee95-2} we know that $D$ is only determined up to an inner automorphism of $G$.  

\medskip 
In \cite{ll09-1}, J.B. Lee and K.B. Lee gave a formula to compute Lefschetz and Nielsen numbers on infra-nilmanifolds. Pick an infra-nilmanifold $\Gamma\backslash G$, determined by the almost-Bieberbach group $\Gamma\subseteq \Aff(G)$ and let $F\subseteq \Aut(G)$ denote the holonomy group of $\Gamma$. We will write $\lie$ for the Lie algebra of $G$. Because $G$ is a nilpotent, connected and simply connected Lie group, the map $\exp:\lie\to G$ will be a diffeomorphism. Therefore, $\Endo(G)$ and $\Endo(\lie)$ are isomorphic and for every endomorphism $A\in \Endo(G)$, we have a unique $A_\ast\in \Endo(\lie)$, which is determined by the relation $A \circ \exp= \exp \circ A_\ast$. This $A_\ast$ will be called the differential of $A$. Of course, $A$ is invertible if and only if $A_\ast$ is invertible.
\begin{theorem}[J.B.\ Lee and K.B.\ Lee \cite{ll09-1}] \label{LeeForm}Let $\Gamma\subseteq \Aff(G)$ be an almost-Bieberbach group with holonomy group $F\subseteq \Aut(G)$. Let $M=\Gamma\backslash G$ be the associated infra-nilmanifold.  If 
 $f:M\ra M$ is a map with affine homotopy lift $(d,D)$, then  
\[L(f)=\frac{1}{\# F}\sum_{A \in F}\det(I-A_\ast D_\ast)\]
and
\[N(f)=\frac{1}{\# F}\sum_{A \in F}|\det(I-A_\ast D_\ast)|.\]
\end{theorem}

We will now list a couple of properties that we will need in this paper.

\medskip

The following lemma can be found in \cite{ddm05-1}. We have adapted the formulation very slightly, but in essence it is the same lemma and it can be proved in a similar way.

\begin{lemma}\label{Lemma Bram}
Suppose that $F\subset \GL_n(\C)$ is a finite group, $D\in \C^{n\times n}$ and for all $A\in F$, there exists a $B\in F$, such that $DA=BD$. Take an arbitrary element $A_1\in F$ and build the sequence $(A_j)_{j\in \N_0}$, such that $DA_i=A_{i+1}D$, for all $i$. Then,
\begin{enumerate}
\item $\forall j\in \N_0: \det(I-A_1D)=\det(I-A_jD).$
\item $\exists l, j\in \N_0: (A_jD)^l=D^l.$
\end{enumerate}
\end{lemma}
%\begin{proof}
%\begin{enumerate}
%\item Note that we have the following equalities:$$\det(I-A_1D)=\det(A_1)\det(A_1^{-1}-D)=\det(I-DA_1)=\det(I-A_2D).$$By using an inductive argument, the result follows.
%\item Since the group $F$ is finite, the sequence $(A_j)_{j\in \N_0}$ becomes periodic from a certain point onwards. So, suppose this period equals $k$ and starts with $A_j$. Then, it is easy to see that$$D^kA_j=D^{k-1}A_{j+1}D=\dots =A_{j+k}D^k=A_jD^k.$$Note that one can show in a similar way that $D^kA_i=A_iD^k$, if $i\geq j$. Now consider $(A_jD)^k=A_jA_{j+1}\dots A_{j+k-1}D^k$. Because every element in the product $B=A_jA_{j+1}\dots A_{j+k-1}$ commutes with $D^k$, the product itself also commutes with $D^k$. Now suppose that $p$ is the order of $B$ in $F$. Then$$(A_jD)^{kp}=(BD^k)^p=B^pD^{kp}=D^{kp}.$$ \end{enumerate}
%\end{proof}

In \cite{dp11-1}, we can find the following theorems.

\begin{theorem}\label{thmRInf}
Let $\Gamma\subseteq \Aff(G)$ be an almost-Bieberbach group with holonomy group $F\subseteq \Aut(G)$. Let $M=\Gamma\backslash G$ be the associated infra-nilmanifold.  If 
 $f:M\to M$ is a map with affine homotopy lift $(d,D)$, then
\[
  R(f)=\infty \iff \exists A \in F \text{ such that } \det(I - A_\ast D_\ast)=0.
\]
\end{theorem}

\begin{theorem}\label{thmN=R}
Let $f$ be a map on an infra-nilmanifold, such that $R(f)<\infty$, then $N(f)=R(f)$.
\end{theorem}

We will also mention the following definition.

\begin{definition}
Let $M$ be an infra-nilmanifold and $f:M\to M$ be a continuous map, with $(d,D)$ as an affine homotopy lift. We say that $f$ is a \textbf{hyperbolic} map if $D_\ast$ has no eigenvalues of modulus $1$. We say that $f$ is \textbf{semi-hyperbolic} if $D_\ast$ has no eigenvalues which are roots of unity.
\end{definition}

This class of (semi-)hyperbolic maps contains for example the class of expanding maps and the class of Anosov diffeomorphisms.

\section{Nielsen periodic point theory}
%For a compact polyhedron $X$ and a continuous self-map $f:X\to X$, one can define certain integers that give information about the fixed points of $f$. One of these integers is the Lefschetz number $L(f)$ which is defined as \[ L(f)= \sum_{i=0}^{{\Dim}\;X} (-1)^i {\Tr} \left( f_{\ast,i} : H_i(X,\Q) \to H_i(X,\Q)\right).\]
%The Lefschetz fixed point theorem states that if $L(f)\neq 0$, $f$ has at least one fixed point. Because the Lefschetz number is only defined in terms of homology groups, it remains invariant under a homotopy and hence, the Lefschetz fixed point theorem remains true for maps homotopic to $f$.
%
%\medskip
%
%Another such integer is the Nielsen number $N(f)$, which is defined as the number of essential fixed point classes of $f$. Because these essential fixed point classes cannot completely vanish under a homotopy, $N(f)$ will be a lower bound for the number of fixed points of any map homotopic to $f$. In general, $N(f)$ will give more information about the fixed points of $f$ than $L(f)$, but the downside is that Nielsen numbers are often much harder to compute than Lefschetz numbers. More information on both numbers can be found in e.g. \cite{brow71-1}, \cite{jian83-1} and \cite{kian89-1}.
%
%\medskip

In this section, we will mostly follow the outline of \cite{jm06-1}. Many of the difficult aspects of Nielsen periodic point theory will disappear when working on infra-nilmanifolds. Therefore, we will try to present all the necessary results in a swift way and skip most of the proofs and unnecessary details. More information about Nielsen periodic point theory in general can be found in \cite{hk97-1},\cite{heat99-1},\cite{jm06-1} or \cite{jian83-1}.

\medskip

When $f^n(x)=x$, we call $x$ a periodic point. If $n$ is the smallest integer for which this holds, $x$ is a periodic point of pure period $n$. We can apply similar techniques as in Nielsen fixed point theory to achieve Nielsen periodic point theory. Just like Nielsen fixed point theory divides $\Fix(f)$ into different fixed point classes, Nielsen periodic point theory divides $\Fix(f^n)$ into different fixed point classes, for all $n>0$ and looks for relations between fixed point classes on different levels. This idea is covered in the following definition.

\begin{definition}\label{defboost}
Let $f:X \to X$ be a self-map. If $\F_k$ is a fixed point class of $f^k$, then $\F_k$ will be contained in a fixed point class $\F_{kn}$ of $(f^k)^n$, for all $n$. We say that $\F_k$ \textbf{boosts} to $\F_{kn}$. On the other hand, we say that $\F_{kn}$ \textbf{reduces} to $\F_k$.
\end{definition}

This idea of boosting a fixed point class also has a more algebraic interpretation. Fix a lifting $\tilde{f}$ of $f$ to the universal covering $(\tilde{X},p)$ of $X$. Then $\tilde{f}$ induces a homomorphism $f_\ast$ on the group of covering transformations by using the following relation:$$f_\ast (\alpha)\circ \tilde{f}=\tilde{f}\circ \alpha.$$Let us denote the set of Reidemeister classes of $f$ by $\mathcal{R}(f)$. Any element of this set will be denoted by the Reidemeister class $[\alpha]$, where $\alpha$ is the coordinate of a lifting $\alpha\circ \tilde{f}$. Let $k,n$ be integers, such that $k|n$. We then define the following boosting function:
$$\gamma_{nk}: \mathcal{R}(f^k)\to \mathcal{R}(f^n):[\alpha]\mapsto [\alpha f_\ast^k(\alpha)f_\ast^{2k}(\alpha)\dots f_\ast^{n-k}(\alpha)].$$The idea behind this boosting function, is the fact that $$(\alpha\circ \tilde{f}^k)^\frac{n}{k}=\alpha f_\ast^k(\alpha)f_\ast^{2k}(\alpha)\dots f_\ast^{n-k}(\alpha)\circ \tilde{f}^n.$$This equality immediately shows that if $\gamma_{nk}([\beta])=[\alpha]$, the fixed point class $p\Fix(\beta\circ \tilde{f}^k)$ will be contained in the fixed point class $p\Fix(\alpha \circ \tilde{f}^n)$. Hence, our algebraic definition makes sense in retrospect to Definition \ref{defboost}.

\medskip

In this paper, we will often make a slight abuse of notation. Whenever we use the expression $[\alpha]_k$, we will simultaneously consider the Reidemeister class $[\alpha]\in \mathcal{R}(f^k)$ and the fixed point class $p(\Fix(\alpha\circ \tilde{f}^k))$. We will also often switch between both of these interpretations, whenever necessary. Note that both interpretations are essentially the same due to the one-to-one correspondence between Reidemeister classes and fixed point classes. This also means that we make a distinction between empty fixed point classes that come from a different Reidemeister classes. In a certain sense, this approach coincides with the idea of \textit{labeled fixed point classes}, in \cite{jian83-1}.

\medskip             

Note that this description only depends on the homomorphism $f_*$. So, when $f$ and $g$ induce the same morphism on the group of covering transformations, then the structure of their periodic point classes will be the same. More specifically, this means that the whole description above (and everything that will follow) is homotopy-invariant.  

\medskip 

When $k|m|n$, an easy computation shows that $\gamma_{nm}\gamma_{mk}=\gamma_{nk}$. Also, $\gamma_{nn}=\Id_{\mathcal{R}(f^n)}$.

\medskip

Actually, to give the precise definition of Nielsen periodic point theory, we need a little more than this definition in terms of classes, namely a definition in terms of orbits.  Define the following map$$\mathcal{R}_f:\mathcal{R}(f^n)\to \mathcal{R}(f^n):[\alpha]\mapsto[f_\ast(\alpha)].$$One can easily see that this map is well-defined and that $\mathcal{R}_f^n=\Id_{\mathcal{R}(f^n)}$. Furthermore, by using the commutativity property of the fixed point index on the maps $f$ and $f^{n-1}$, it is clear that this map preserves the index of the associated fixed point classes. By identifying $[\alpha]$ with $[f_\ast(\alpha)]$ in $\mathcal{R}(f^n)$, for all $\alpha$, we find the quotient set $\mathcal{OR}(f^n)$ of orbits of Reidemeister classes. Since the index is preserved in every orbit, it makes sense to talk about essential and inessential orbits. One can also notice that boosting functions make sense in terms of orbits, so we can talk about reducible and irreducible orbits, depending on whether they have a pre-image under a boosting function or not.

\begin{lemma} [\cite{jm06-1}, 5.1.13]\label{lemessirred}
If $\mathcal{A}\in \mathcal{OR}(f^n)$ is essential and irreducible, then this orbit contains at least $n$ periodic points of period $n$.
\end{lemma}

This lemma gives us the idea for the following definition.

\begin{definition}
We define the \textbf{prime Nielsen-Jiang periodic number} $NP_n(f)$ as $$n \times \textrm{(number of irreducible essential orbits in }\mathcal{OR}(f^n)).$$
\end{definition}

If $P_n(f)$ is the set of periodic points of $f$ of pure period $n$, then Lemma \ref{lemessirred} ensures us that $NP_n(f)$ is a homotopy-invariant lower bound of $\# P_n(f)$.

\medskip

We would also like to find a similar lower bound for $\# \Fix(f^n)$. Pick an arbitrary $\mathcal{A}\in \mathcal{OR}(f^n)$. We define the depth $d(\mathcal{A})$ to be the least divisor $k$ of $n$, such that $\mathcal{A}\in \textrm{Im}(\gamma_{nk})$.

\begin{definition}
Let $n$ be a fixed positive integer. A subset $\mathcal{PS}\subset \bigcup_{k|n}\mathcal{OR}(f^k)$ is called a \textbf{preceding system} if every essential orbit $\mathcal{A}$ in $\bigcup_{k|n}\mathcal{OR}(f^k)$ is preceded by an element of $\mathcal{PS}$. Such a preceding system is called \textbf{minimal} if the number $\sum_{\mathcal{A}\in \mathcal{PS}}d(\mathcal{A})$ is minimal.
\end{definition}

\begin{definition}
The \textbf{full Nielsen-Jiang periodic number} $NF_n(f)$ is defined as $\sum_{\mathcal{A}\in \mathcal{PS}}d(\mathcal{A})$, where $\mathcal{PS}$ is a minimal preceding system.
\end{definition}

\begin{theorem}[\cite{jm06-1}, 5.1.18]
$NF_n(f)$ is a homotopy-invariant lower bound for the number $\# \Fix(f^n)$.
\end{theorem}
Note that every preceding system must contain every essential irreducible orbit of $\mathcal{OR}(f^n)$. Since every of these orbits has a depth of $n$, we know that the following inequality holds:$$\sum_{k|n}NP_k(f)\leq NF_n(f).$$

An important definition that gives some structure to the boosting and reducing relations is the following.

\begin{definition}
A self-map $f:X\to X$ will be called \textbf{essentially reducible} if, for all $n,k$, essential fixed point classes of $f^{kn}$ can only reduce to essential fixed point classes of $f^k$. A space $X$ is called essentially reducible if every self-map $f:X\to X$ is essentially reducible.
\end{definition}

It can be shown that the fixed point classes for maps on infra-nilmanifolds always have this nice structure for their boosting and reducing relations.

\begin{theorem}[\cite{lz07-1}]\label{thmleezhao}
Infra-nilmanifolds are essentially reducible.
\end{theorem}

A nice consequence of being essentially reducible, is the following lemma.

\begin{lemma}[\cite{jm06-1}, 5.1.22]\label{lemNPNF}
If $f:X\to X$ is essentially reducible, then it has a unique minimal preceding system, namely the set of all the essential irreducible orbits in $\bigcup_{k|n}\mathcal{OR}(f^k)$. As a consequence, the following equality holds:$$\sum_{k|n}NP_k(f)=NF_n(f).$$
\end{lemma}

Of course, by using the M\"{o}bius inversion formula, we can also write$$NP_n(f)=\sum_{k|n}\mu\left(\frac{n}{k}\right)NF_k(f),$$where $\mu$ denotes the M\"{o}bius function.

\medskip

As a generalization of being essentially reducible, we can define two other structures on the boosting and reducing relations.

\begin{definition}
A map $f:X\to X$ is called \textbf{essentially reducible to the greatest common divisor} (GCD) if it is essentially reducible and if for every essential fixed point class $[\alpha]_n$ that reduces to both $[\beta]_k$ and $[\gamma]_l$, there exists a fixed point class $[\delta]_d$, with $d=\gcd(k,l)$, such that $[\alpha]_n$ reduces to $[\delta]_d$. If this holds for every self-map on $X$, we will say that $X$ is essentially reducible to the GCD.
\end{definition}

An easy consequence of this definition is the following lemma.

\begin{lemma}[\cite{jm06-1}, 5.1.26]\label{lemessredgcd}
If $f:X\to X$ is essentially reducible to the GCD, then every essential fixed point class $[\alpha]_n$ in $\mathcal{R}(f^n)$ is preceded by a unique irreducible essential fixed point class $[\beta]_k$. Moreover, $d([\alpha]_n)=k$.
\end{lemma}

By the length $l([\alpha]_n)$, we mean the minimal number $l|n$, such that $\mathcal{R}_f^l([\alpha]_n)=[\alpha]_n$. Alternatively, this is the number of fixed point classes in an orbit $\mathcal{A}\in \mathcal{OR}(f^n)$.

\medskip

It is immediately clear that $d([\alpha])\geq l([\alpha])$, because every class in an orbit that reduces to depth $d$ will be a fixed point of the map $\mathcal{R}_f^d$.

\begin{definition}
A map $f:X\to X$ is called \textbf{essentially toral} if it is essentially reducible and if the following two conditions are fulfilled:\begin{enumerate}
\item For every essential fixed point class in $\mathcal{R}(f^n)$, the length and depth coincide.
\item If $[\alpha]_n$ is essential and $\gamma_{nk}([\beta]_k)=\gamma_{nk}([\gamma]_k)=[\alpha]_n$, then $[\beta]_k=[\gamma]_k$.
\end{enumerate}
If this holds for every self-map on $X$, we will say that $X$ is essentially toral.
\end{definition}

Because lengths and depths coincide, the following lemma follows easily.
\begin{lemma}[\cite{jm06-1}, 5.1.30]\label{lemclasses}
If $f:X\to X$ is essentially toral, then $NP_n(f)$ equals the number of irreducible essential fixed point classes in $\mathcal{R}(f^n)$.
\end{lemma}
This lemma actually tells us that if we are working on an essentially toral space, we are free to replace the orbit theory by a theory in terms of classes. By combining Lemma \ref{lemNPNF} and Lemma \ref{lemclasses}, the following can  also be easily deduced.

\begin{corollary}\label{coresstor}
If $f:X\to X$ is essentially toral, then $NF_n(f)$ equals the number of irreducible essential fixed point classes in $\bigcup_{k|n}\mathcal{R}(f^k)$.
\end{corollary}

In \cite{hk97-1}, the following theorem is proved.

\begin{theorem}
Nilmanifolds are essentially reducible to the GCD and essentially toral.
\end{theorem}

Note that they actually proved a more general version of this theorem, as they showed that the theorem above also holds for solvmanifolds.

\begin{definition}
A map $f:X\to X$ is called \textbf{weakly Jiang} if $N(f)=0$ or $N(f)=R(f)$. This means that all fixed point classes are simultaneously essential or inessential.
\end{definition}

\begin{theorem}[\cite{hk97-1}, Theorem 5.1]\label{thmhk}
Suppose that $X$ is essentially toral and essentially irreducible to the GCD. If $f:X\to X$ is a map such that $f^n$ is weakly Jiang and $N(f^n)\neq 0$, then $$NF_n(f)=N(f^n)$$and the same formula holds for every divisor of $n$.
\end{theorem}
The idea behind this proof is very simple. Since every fixed point class at level $n$ is essential, by Lemma \ref{lemessredgcd}, we know that every such class is preceded by a unique irreducible essential class. On the other hand, every irreducible essential fixed point class in $\bigcup_{k|n}\mathcal{R}(f^k)$ has to boost essentially, since there are simply no inessential fixed point classes to boost to at level $n$. Because of these observations, there exist a bijection between the essential fixed point classes at level $n$ and the irreducible essential fixed point classes in $\bigcup_{k|n}\mathcal{R}(f^k)$. Corollary~\ref{coresstor} then proves the theorem.

\medskip

It is known that every map on a nilmanifold is weakly Jiang, due to the result of Anosov (\cite{anos85-1}) or Fadell and Husseini (\cite{fh86-1}). Consequently, Theorem \ref{thmhk} holds for nilmanifolds. Unfortunately, not every map on an infra-nilmanifold is weakly Jiang. Even on the Klein bottle, the smallest example of an infra-nilmanifold which is not a nilmanifold, it is possible to find counterexamples.

\begin{example}\label{Example not weakly Jiang}
Suppose we have the following presentation of the Klein bottle group: $$<\alpha,\beta | \alpha \beta =\beta^{-1} \alpha>.$$Let $k\neq 1$ be odd. Now, let $f_*:\alpha\mapsto \alpha^k, \textrm{ } \beta\mapsto \beta^{-1}$ be the induced morphism for a map $f$ on the Klein bottle. One can check that this morphism indeed induces a map on the Klein bottle, for which it holds that $R(f)=\infty$, while $N(f)\neq 0$.
\end{example}
An algebraic argument for the fact that maps on nilmanifolds are weakly Jiang, while maps on infra-nilmanifolds are generally not, can be found by combining Theorem \ref{LeeForm} with Theorem~\ref{thmRInf} and Theorem \ref{thmN=R}. When working on nilmanifolds, the formula in Theorem \ref{LeeForm} reduces to a single determinant. By Theorem \ref{thmRInf}, we know that this determinant will be equal to $0$ (and hence $N(f)=0$) if and only if $R(f)=\infty$. By combining this fact with Theorem \ref{thmN=R}, it follows that nilmanifolds are weakly Jiang. When working on infra-nilmanifolds, the sum generally consists of multiple determinants. Therefore, it is possible that some of these determinants are $0$ and some are not. If this is the case, a similar argument as before will show that the map is not weakly Jiang, as $R(f)=\infty$, while $N(f)\neq 0$.
\section{Structure on the periodic point classes of infra-nilmanifolds}
In this section, we will show that infra-nilmanifolds are both essentially reducible to the GCD and essentially toral. As a result of these structural properties, we will be able to show a theorem similar to Theorem \ref{thmhk} for semi-hyperbolic maps on infra-nilmanifolds.

\medskip

We will prove both of these structural properties for affine maps on infra-nilmanifolds and because the theory described in the previous section is homotopy-invariant, this will be sufficient. As already mentioned before, affine maps are often much easier to deal with. This fact is exemplified in the following proposition, which can be found in \cite{fl13-1}.

\begin{proposition}\label{propaff}
If $\overline{(d,D)}:M\to M$ is an affine map on an infra-nilmanifold, then every non-empty fixed point class is path-connected and
\begin{enumerate}
\item every essential fixed point class of $\overline{(d,D)}$ consists of exactly one point.
\item every non-essential fixed point class of $\overline{(d,D)}$ is empty or consists of infinitely many points.
\end{enumerate} 
\end{proposition}

Now we can prove the two main theorems of this section.

\begin{theorem}\label{thmessredgcd}
Infra-nilmanifolds are essentially reducible to the GCD.
\end{theorem}
\begin{proof}
By Theorem \ref{thmleezhao}, we already know that infra-nilmanifolds are essentially reducible.
\medskip
It is known from \cite{ll06-1} that every almost-Bieberbach group $\Gamma$ has a fully characteristic subgroup $\Lambda$ of finite index, such that $\Lambda\subset G$. Therefore, every infra-nilmanifold of the form $\Gamma\backslash G$ is finitely covered by a nilmanifold $\Lambda\backslash G$, such that every continuous map $f:\Gamma\backslash G\to \Gamma\backslash G$ can be lifted to a map $\overline{f}:\Lambda\backslash G\to \Lambda\backslash G$. All in all, we have the following commuting diagram, where $\beta_n$ is a covering transformation and $\overline{\beta}_n$ is the natural projection of $\beta_n$ into $\Gamma/ \Lambda$.
\begin{displaymath}
    \xymatrix{ G \ar[dd]_p \ar[dr]_{p'} \ar[rrrrr]^{\beta_n \tilde{f}^n}& & & & & G \ar[dd]^{p} \ar[dl]^{p'} \\
    & \Lambda\backslash G \ar[rrr]^{\overline{\beta}_n\overline{f}^n} \ar[dl]_{\overline{p}} & & & \Lambda\backslash G \ar[dr]^{\overline{p}} & \\
               \Gamma\backslash G \ar[rrrrr]^{f^n}& & & &  & \Gamma\backslash G  }.
\end{displaymath} 

Suppose that $g$ is the affine map on $\Gamma\backslash G$ that is induced by an affine homotopy lift $\tilde{g}$ of $f$. Let $p(\Fix(\beta_n\tilde{g}^n))$ be an essential fixed point class on level $n$, such that, for $r,s|n$, this fixed point class reduces to $p(\Fix(\beta_r\tilde{g}^r))$ and $p(\Fix(\beta_s\tilde{g}^s))$.

\medskip

Because of Proposition \ref{propaff}, we know that there exists $x\in \Fix(g^n)$, such that all these fixed point classes are equal to the set $\{x\}$. The fixed point index is a local property and a covering map is a local homeomorphism, hence, the fixed point class $p'(\Fix(\beta_n\tilde{g}^n))$ is also essential. By using Proposition \ref{propaff} again, we know this fixed point class will consist of one point, namely a $\overline{x}\in \overline{p}^{-1}(x)$. By a similar reasoning, there will exist $\gamma_r,\gamma_s \in \Gamma$ and accordingly, $\overline{\gamma_r},\overline{\gamma_s} \in \Gamma /\Lambda$ such that $$p'(\Fix(\beta_r\tilde{g}^r))=\{\overline{\gamma_r}\cdot \overline{ x}\}\textrm{ and }p'(\Fix(\beta_s\tilde{g}^s))=\{\overline{\gamma_s}\cdot \overline{ x}\}.$$An easy calculation then shows that $$p'(\Fix(\gamma_r^{-1}\beta_r g_\ast^r(\gamma_r)\tilde{g}^r))=\{\overline{x}\}\textrm{ and }p'(\Fix(\gamma_s^{-1}\beta_sg_\ast^s(\gamma_s)\tilde{g}^s))=\{\overline{x}\}.$$This actually means that if we choose good representatives in the Reidemeister classes of $[\beta_r]_r$ and $[\beta_s]_s$, $p'(\Fix(\beta_n\tilde{g}^n))$ will reduce to both $p'(\Fix(\beta_r\tilde{g}^r))$ and $p'(\Fix(\beta_s\tilde{g}^s))$ on our nilmanifold $\Lambda\backslash G$. Since nilmanifolds are known to be essentially reducible to the GCD, there exists a $\beta_d$, with $d=\gcd(r,s)$, such that $p'(\Fix(\beta_r\tilde{g}^r))$ and $p'(\Fix(\beta_s\tilde{g}^s))$ both reduce to  $p'(\Fix(\beta_d\tilde{g}^d))$. By applying $\overline{p}$ to this fixed point class, the statement is proved.
\end{proof}

\begin{theorem}\label{thmesstor}
Infra-nilmanifolds are essentially toral.
\end{theorem}
\begin{proof}
Again, we already know that infra-nilmanifolds are essentially reducible and again, by homotopy-invariance, it suffices to prove this theorem for affine maps $g$.

\medskip

Let $[\alpha]_n$ be an essential fixed point class of $g^n$. Since we already know that $d([\alpha]_n)\geq l([\alpha]_n)$, we only need to prove that the strict inequality is impossible. So, suppose that $d=d([\alpha]_n)> l([\alpha]_n)=l$. Because of Proposition \ref{propaff}, we know that there exists $x\in \Fix(g^n)$ such that $\{x\}$ is the fixed point class associated to $[\alpha]_n$. Furthermore, $\{g(x)\}$ will be the fixed point class associated to $\mathcal{R}_g([\alpha]_n)$. By definition and because there is only one fixed point in each essential fixed point class, $g^l(x)=x$. Therefore, $[\alpha]_n$ reduces to a fixed point class on level $l$, which is a contradiction to the fact that $d>l$. This proves the first condition.

\medskip

If $[\beta]_k$ and $[\gamma]_k$ are both boosted to $[\alpha]_n$, then we know that they are both essential fixed point classes. Hence, they both have the set $\{x\}$ as associated fixed point class, which means that $[\beta]_k=[\gamma]_k$. This proves the second condition of essential torality.
\end{proof}

Now we will use these newly obtained structural properties for infra-nilmanifolds to establish a few results concerning Nielsen periodic points. We start with the following definition.

\begin{definition}
We say that an essential fixed point $[\alpha]_k$ is \textbf{(in)essentially boosted to level $n$}, if $[\alpha]_k$ is boosted to an (in)essential fixed point class $[\beta]_{n}$. 
\end{definition}

Let us denote the set of all irreducible fixed point classes which are inessentially boosted to level $n$ for a continuous self-map $f$ by $IIB_n(f)$. Note that this is a subset of $\bigcup_{k|n}\mathcal{R}(f^k)$, since this set contains all fixed point classes on all levels that will boost to level $n$.

\begin{proposition}\label{propIIB}
Whenever a map $f$ is essentially reducible to the GCD and essentially toral, we have that
$$NF_n(f)=N(f^n)+\#IIB_n(f).$$
\end{proposition}
\begin{proof}
When a map $f$ is essentially toral, we know by Corollary \ref{coresstor} that $NF_n(f)$ equals the number of irreducible essential classes in $\bigcup_{k|n}\mathcal{R}(f^k)$. Now, pick an arbitrary irreducible essential class. We can distinguish two disjoint cases. 

\medskip

On the one hand, suppose this class boosts essentially to level $n$. As $f$ is essentially reducible to the GCD, we can apply Lemma \ref{lemessredgcd} and we know that every essential fixed point class reduces to a unique irreducible essential fixed point class. This means that there is a bijection between the irreducible essential classes that are essentially boosted to level $n$ and the essential fixed point classes of $\mathcal{R}(f^n)$.

\medskip

If, on the other hand, our class boosts inessentially to level $n$, it belongs to $IIB_n(f)$. Since both cases are disjoint, the equality follows.
\end{proof}

It is quite easy to see that this proposition is a generalization of Theorem \ref{thmhk}. In fact the proof is a slightly adapted version where we take inessential boosting into account.

\begin{theorem}\label{thmNF=Nf)}
When $f$ is a semi-hyperbolic map on an infra-nilmanifold, then for all $n>0$ $$NF_n(f)=N(f^n).$$
\end{theorem}
\begin{proof}
Suppose that $(d,D)$ is an affine homotopy lift of $f$. By combining Theorem \ref{thmRInf} and Theorem \ref{thmN=R} we know that every fixed point class on level $n$ is essential if and only if for all $A\in F$ (where $F$ is the holonomy group of our infra-nilmanifold),$$\det(I-A_\ast D_\ast^n)\neq 0.$$By Lemma \ref{Lemma Bram}, we know that there exists $B \in F$ and an integer $l$, such that $$(B_\ast D_\ast^n)^l=D_\ast^{ln} \textrm{ and } \det(I-A_\ast D_\ast^n)=\det(I-B_\ast D_\ast^n).$$Note that $\det(I-B_\ast D_\ast^n)=0$ implies that $B_\ast D_\ast^n$ has an eigenvalue $1$, but this would mean that $D_\ast^{ln}$ had an eigenvalue $1$, which is in contradiction with the fact that $f$ is semi-hyperbolic. Therefore, we know that every fixed point class on level $n$ is essential, which implies that $IIB_n(f)$ is the empty set. The theorem then follows from Proposition \ref{propIIB}.
\end{proof}

Note that the proof of this theorem actually also proves the following proposition, since we proved that every fixed point class on every level is essential.

\begin{proposition}
When $f$ is a semi-hyperbolic map on an infra-nilmanifold, then for all $n>0$, $f^n$ is a weakly Jiang map.
\end{proposition}

With this proposition in mind, one can easily see that Theorem \ref{thmNF=Nf)} is a special case of Theorem \ref{thmhk}.

\medskip

Later on, in the last section, we will show, under mild conditions, that semi-hyperbolic maps are the only maps for which a non-trivial equality $NF_n(f)=N(f^n)$ holds.

\medskip

Theorem \ref{thmNF=Nf)} actually has a nice corollary in the area of dynamical zeta functions. By $N_f(z)$, we mean the Nielsen zeta function, as defined in \cite{fels00-2} ,\cite{fels88-1} or \cite{fp85-1}. In \cite{fels00-2}, the following definition of the  \textbf{minimal dynamical zeta function} can be found: $$NF_f(z)=\exp\left(\sum_{k=1}^\infty \frac{NF_k(f)z^k}{k}\right).$$

%By the main result of \cite{jezi03-01}, we actually know that for every self-map $f$ on a compact manifold of $\dim\geq 3$, there exists a map $g$ in the homotopy class of $f$ such that $NF_f(z)$ coincides with the following zeta function, which was also defined in \cite{fels88-1}
%$$M_g(z)=\exp\left(\sum_{k=1}^\infty \frac{\#\Fix(g^k)z^k}{k}\right).$$

We now have the following corollary.

\begin{corollary}
Let $f$ be a semi-hyperbolic map on an infra-nilmanifold, then $N_f(z)=~NF_f(z)$.
\end{corollary}

By using the main result of \cite{dd13-2}, which states that Nielsen zeta functions are rational for self-maps on infra-nilmanifolds, we can also conclude the following.

\begin{corollary}
Let $f$ be a semi-hyperbolic map on an infra-nilmanifold, then $NF_f(z)$ is a rational function.
\end{corollary}

\section{A method for computing $NF_n(f)$}
In theory, we are now capable to compute $NF_n(f)$, due to Proposition \ref{propIIB}. By using the standard formula for Nielsen numbers for maps on infra-nilmanifolds (Theorem \ref{LeeForm}), the computation of $N(f^n)$ becomes very simple and therefore, the only thing left to check is how many fixed point classes lie in $IIB_n(f)$. \medskip

In some cases, for example for semi-hyperbolic maps, the computation of $\# IIB_n(f)$ becomes trivial. However, in a more general setting, this number can be a very tedious thing to compute. In this section, we will try to develop a method to make this computation a bit easier.

\subsection{$\sim_f$-equivalence classes}
We start this section with the following definition.
\begin{definition}
Let $f: \Gamma\backslash G \to \Gamma\backslash G$ be a continuous map, such that $F$ is the holonomy group of $\Gamma$. We will say that $A, B \in F$ are \textbf{$f$-conjugated}, if there exist $a,b\in G$ and $\gamma\in \Gamma$ such that $(a,A)$ and $(b, B)$ are elements of $\Gamma$ and $$\gamma\circ (a,A) \circ f_*(\gamma^{-1})=(b,B).$$We will write $A\sim_f B.$
\end{definition}

An alternative for this definition is given in the following lemma. In general, the definition will be more useful when one quickly wants to find elements that are $f$-conjugated. The lemma below is often more useful when it comes to finding properties of the $\sim_f$-relation. 
\begin{lemma}\label{lemma alternatief definitie}
Let $f: \Gamma\backslash G \to \Gamma\backslash G$ be a continuous map, such that $F$ is the holonomy group of $\Gamma$. Then $A\sim_f B$ if and only if for all $(a,A)\in \Gamma$, there exist $(b,B), \gamma \in \Gamma$, such that$$\gamma\circ (a,A) \circ f_*(\gamma^{-1})=(b,B).$$
\end{lemma} 
\begin{proof}
One direction is obvious. For the other direction, pick an arbitrary $(a,A)\in \Gamma$ and suppose that $A\sim_f B$. This means that there exist $a_0, b_0 \in G$ and $\gamma_0 \in \Gamma$, such that $(a_0,A),(b_0,B) \in \Gamma$ and $$\gamma_0\circ (a_0,A) \circ f_*(\gamma_0^{-1})=(b_0,B).$$Then $$\gamma_0\circ (a,A) \circ f_*(\gamma_0^{-1})=\left(\gamma_0\circ (a_0,A) \circ f_*(\gamma_0^{-1})\right) \circ \left(f_*(\gamma_0) \circ (A^{-1}(a_0^{-1}a), \Id) \circ f_*(\gamma_0^{-1})\right).$$As $\Gamma \cap G$ is a normal divisor of $\Gamma$, there exists a $(c,\Id)\in \Gamma\cap G$, such that $$\gamma_0\circ (a,A) \circ f_*(\gamma_0^{-1})=(b_0,B)\circ (c,\Id).$$
\end{proof}

A simple consequence of the previous lemma is the following.
\begin{corollary}
$\sim_f$ is an equivalence relation.
\end{corollary}
\begin{proof}
The fact that this relation is reflexive and symmetric is easy to see, so the only thing left to prove is the transitivity. Suppose that $A\sim_f B$ and $B\sim _f C$. By definition, there exist $a, b \in G$ and $\gamma_1 \in \Gamma$, such that $(a,A), (b,B)\in \Gamma$ and $$\gamma_1\circ (a,A) \circ f_*(\gamma_1^{-1})=(b,B).$$By Lemma \ref{lemma alternatief definitie}, we know there exist $\gamma_2, (c,C) \in \Gamma$, such that $$\gamma_2\circ (b,B) \circ f_*(\gamma_2^{-1})=(c,C).$$By combining both equations, we see$$(\gamma_2 \circ \gamma_1)\circ (a,A)\circ f_*((\gamma_2\circ \gamma_1)^{-1})=(c,C),$$which means that $A\sim_f C$.
\end{proof}

The fact that $\sim_f$ is an equivalence relation implies that we can partition $F$ into \textbf{$\sim_f$-equivalence classes}. There is an even more convenient way to look at the $\sim_f$-equivalence classes for which we only need to work in the holonomy group $F$ of our infra-nilmanifold. This will be the ideal tool to compute these classes in a more effective way.

\begin{definition}
By $f_\#(\Id)$, we mean the set of all $A\in F$, such that there exist $(g,\Id), (a,A)\in \Gamma$, for which $f_*(g,\Id)=(a,A)$. Analogously, $f_\#(C)$ is the set of all $A\in F$, such that there exist $(c,C), (a,A)\in \Gamma$, for which $f_*(c,C)=(a,A)$.
\end{definition}

Note that it is known that $\Gamma\cap G$ is finitely generated. With this in mind, we can deduce the following lemma. 

\begin{lemma}\label{lemma f-equivalentieklassen}
Pick an arbitrary $(c,C)\in \Gamma$. Again, $p:\Aff(G)=G\semi \Aut(G) \to \Aut(G)$ denotes the natural projection onto the second factor of the semi-direct product. Suppose that $(g_i,\Id)_{i=1}^n$ is a set of generators for $\Gamma\cap G$. Then we can describe $f_\#(\Id)$ and $f_\#(C)$ as follows:
\begin{itemize}
\item $f_\#(\Id)=\grp\{p(f_*(g_i,\Id))\}.$
\item $f_\#(C)=p(f_*(c,C))f_\#(\Id)=f_\#(\Id)p(f_*(c,C)).$
\end{itemize}

\end{lemma}
\begin{proof}
It is clear that $f_\#(\Id)$ contains all elements $p(f_*(g_i,\Id))$ and it is also clear that $f_*(\Id)$ is precisely the set $p(f_*(\Gamma\cap G))=p(f_*(\grp \{ (g_i, \Id)\ \| \ i=1\dots n\}))$. As $p\circ f_*$ is a morphism, this will be equal to $\grp\{p(f_*( g_i, \Id)) \ \| \ i=1\dots n\}$, which proves the first statement.

\medskip

Take an arbitrary element of the form $(c_1, C)\in \Gamma$. It is clear that $p(f_*(c_1,C))\in f_\#(C)$. Now, an easy computation shows that $$p(f_*(c_1,C))=p(f_*(c,C))p(f_*(c,C)^{-1})p(f_*(c_1,C))=p(f_*(c,C))p(f_*(C^{-1}(c^{-1}c_1),\Id)).$$As $p(f_*(C^{-1}(c^{-1}c_1),\Id))\in f_\#(\Id)$, the first equality of second statement is proved. The second equality can be proved in a similar way, by multiplying with  $p(f_*(c,C)^{-1})p(f_*(c,C))$ on the right.
\end{proof}

As a side remark, note that an easy consequence of this lemma is the fact that if $p\circ f_*:\Gamma\to F$ is a surjective morphism, then $f_\#(\Id)$ will be a normal divisor of $F$.

\medskip 

By using this lemma, we can derive an easier way of determining $\sim_f$-equivalence classes.

\begin{proposition}
Suppose $A,B$ are elements in $F$. Then, $A\sim_f B$ if and only if there exists a $C\in F$, such that $B\in CAf_\#(C)^{-1}$. Here $f_\#(C)^{-1}$ denotes the set of all inverses of elements in $f_\#(C)$, or equivalently $f_\#(C^{-1})$.
\end{proposition}
\begin{proof}
One direction is obvious. For the other direction, suppose that there exists a $C\in F$, such that $B\in CAf_\#(C)^{-1}$. Then, there exist $a,c\in G$, such that $(c,C), (a,A)\in \Gamma$. By Lemma~\ref{lemma f-equivalentieklassen}, any element in $f_\#(C)^{-1}$ will come from an element of the form $f_*(c,C)^{-1}f_*(g,\Id)$, with $(g,\Id)\in \Gamma$. As a result, we find that there exists a $b\in G$, such that $$(c,C)(a,A)f_*(c,C)^{-1}f_*(g,\Id)=(b,B).$$Note that $(b,B)$ will also be an element of $\Gamma$. By multiplying both sides on the left with $(g^{-1},\Id)$ (which is also in $\Gamma$), we get$$ (g^{-1},\Id)(c,C)(a,A)f_*((g^{-1},\Id)(c,C))^{-1}=(g^{-1},\Id)(b,B)=(g^{-1}b,B).$$This proves that $A\sim_f B$.

\medskip

In order to see that $f_\#(C^{-1})=f_\#(C)^{-1}$, note that Lemma \ref{lemma f-equivalentieklassen} tells us that $f_\#(\Id)$ is a group and that $f_\#(C)=p(f_*(c,C))f_\#(\Id)=f_\#(\Id)p(f_*(c,C))$, from which this fact follows immediately, as $$f_\#(C)^{-1}=f_\#(\Id)^{-1}p(f_*(c,C))^{-1}=f_\#(\Id)p(f_*(c,C)^{-1})=f_\#(C^{-1}).$$
\end{proof}

\begin{corollary}
The $\sim_f$-equivalence class of $A$ equals the set $$\bigcup_{C\in F} CAf_\#(C)^{-1}.$$
\end{corollary}

\begin{corollary}\label{corollary D invertible}
Let $(d,D)$ be an affine homotopy lift of a continuous map $f$ on an infra-nilmanifold $\Gamma\backslash G$. When $D$ is invertible in $\Endo(G)$, for any $C\in F$, $f_\#(C)$ will be a singleton.
\end{corollary}
\begin{proof}
Take an arbitrary element $(g,\Id)$ of $\Gamma\cap G$. Then $f_*(g,\Id)\circ (d,D)=(d,D)\circ (g,\Id)$ and hence also, $p(f_*(g,\Id))\circ D=D$. As $D$ is invertible, $p(f_*(g,\Id))=\Id$. As $(g,\Id)$ was chosen arbitrarily, this means $f_\#(\Id)=\{\Id\}$. By Lemma \ref{lemma f-equivalentieklassen}, $f_\#(C)$ is also a singleton.  
\end{proof}

\subsection{Properties of $\sim_f$-equivalence classes}
A first sign that shows that elements in the same $\sim_f$-equivalence class are strongly connected, can be found in the following lemma.

\begin{lemma}\label{lemdet}
When $A\sim_f B$ and $(d,D)$ is an affine homotopy lift of $f$, then $$\det(I-A_\ast D_\ast)=\det(I-B_\ast D_\ast).$$
\end{lemma}
\begin{proof}
As $A\sim_f B$, there exists a $(c,C) \in \Gamma$, such that $$(c,C)\circ(a,A)\circ f_*(c, C)^{-1}=(b,B).$$Of course, we can compose both sides with $(d,D)$. As a result, we get the following equality$$(c,C)\circ(a,A)\circ (d,D)\circ (c,C)^{-1}=(b,B)\circ(d,D).$$Therefore, we have $CADC^{-1}=BD$ in $\Aut(G)$. From this, the statement follows easily.
\end{proof}

The following theorem is in a certain sense the heart of our computational method. It splits the formula from Theorem \ref{LeeForm} into several parts, one for each $\sim_f$-equivalence class. The proof is heavily influenced by the proofs in \cite{kll05-1} and \cite{ll06-1}.

\begin{theorem}\label{thmNAf}
Let $f:\Gamma\backslash G \to \Gamma\backslash G$ be a continuous map on an infra-nilmanifold, with affine homotopy lift $(d,D)$ and holonomy group $F$. Let $(G,p)$ be a universal covering of $\Gamma\backslash G$, such that $\tilde{f}$ is a reference lifting of $f$. If we fix $A \in F$, then, the number of essential fixed point classes that can be written as $p(\Fix((a,A)\circ \tilde{f}))$, which we will denote by $N_A(f)$, equals $$\frac{1}{\#F}\sum_{B\sim_f A}|\det(I-B_\ast D_\ast)|.$$
\end{theorem}
\begin{proof}
From now on, $\Lambda$ will be the fully characteristic subgroup of $\Gamma$ described in \cite{ll06-1}. Similarly to Theorem \ref{thmessredgcd}, we have the following commuting diagram:

\begin{displaymath}
    \xymatrix{ G \ar[dd]_p \ar[dr]_{p'} \ar[rrrrr]^{(a,A) \circ \tilde{f}}& & & & & G \ar[dd]^{p} \ar[dl]^{p'} \\
    & \Lambda\backslash G \ar[rrr]^{\overline{(a,A)} \circ \overline{f}} \ar[dl]_{\overline{p}} & & & \Lambda\backslash G \ar[dr]^{\overline{p}} & \\
               \Gamma\backslash G \ar[rrrrr]^{f}& & & &  & \Gamma\backslash G  }.
\end{displaymath} 

For $\alpha \in \Gamma$, we will denote the Reidemeister class of $\alpha$ in $\Gamma$ by $[\alpha]_\Gamma$. Now, define an equivalence relation $\sim{_\Lambda}$ on $\Gamma$ as follows:$$\alpha\sim{_\Lambda}\beta \text{ iff } \exists \lambda\in \Lambda: \beta=\lambda\circ \alpha\circ f_*(\lambda)^{-1}.$$In a similar way as before, $[\alpha]_\Lambda$ will denote the equivalence class with respect to $\sim_\Lambda$ that contains $\alpha$. It is straightforward to prove that $\beta\in [\alpha]_\Lambda$ implies that $p'(\Fix(\beta\circ \tilde{f}))=p'(\Fix(\alpha\circ \tilde{f}))$. In a similar way, one can prove that $\beta\not \in [\alpha]_\Lambda$ implies that $p'(\Fix(\beta\circ \tilde{f}))\cap p'(\Fix(\alpha\circ \tilde{f}))=\emptyset$. Note that this can also mean that $p'(\Fix(\beta\circ \tilde{f}))$ and $ p'(\Fix(\alpha\circ \tilde{f}))$ are fixed point classes for different maps on the nilmanifold $\Lambda\backslash G$. Now, by labeling the possibly empty fixed point sets, we can say that there is a one-to-one relation between the sets $[\alpha]_\Lambda$ and the fixed point classes $p'(\Fix(\alpha\circ \tilde{f}))$ of liftings of $f$ to $\Lambda\backslash G$. Let us denote the set of $\sim_\Lambda$-equivalence classes by $\mathcal{R}_\Lambda(f)$ and the set of Reidemeister classes of $f$ by $\mathcal{R}(f)$. As $\Lambda\leq \Gamma$, we know that the map $$\Psi:\mathcal{R}_\Lambda(f)\to \mathcal{R}(f): [\alpha]_\Lambda\mapsto [\alpha]_\Gamma$$is a well-defined function, which is clearly surjective. We also know that $p(\Fix(\alpha\circ \tilde{f}))$ is essential if and only if $p'(\Fix(\alpha\circ \tilde{f}))$ is essential, because the fixed point index is a local property and $\overline{p}$ is a local homeomorphism. Hence, when $[\alpha]_\Gamma$ is (in)essential, then every element in $\Psi^{-1}([\alpha]_\Gamma)$ corresponds to an (in)essential fixed point class of a lift of $f$ to $\Lambda\backslash G$. Because of this property and the fact that $\Psi$ is surjective and well-defined, we know that \begin{equation}\label{ineq}
N_A(f)\leq \sum_{\begin{subarray}{c}
\overline{(b,B)}\in \Gamma / \Lambda \\ B\sim_f A
\end{subarray}}N(\overline{(b,B)}\circ \overline{f}).
\end{equation}

\medskip

Suppose that $\mathbb{F}$ is a fixed point class of the desired form. Then $$\mathbb{F}=p(\Fix((b,B)\circ \tilde{f})=[(b,B)]_\Gamma,$$with $B\sim_f A$.
When $\F$ is an inessential fixed point class, $[(b,B)]_\Lambda$ corresponds to an inessential fixed point class of the map $\overline{(b,B)}\circ \overline{f}$ on the nilmanifold $\Lambda\backslash G$. Due to the main result from \cite{anos85-1} or \cite{fh86-1}, we now know that every fixed point class of $\overline{(b,B)}\circ \overline{f}$ is inessential, so that $N(\overline{(b,B)}\circ \overline{f})=0$. This also means that $\det(I-B_\ast D_\ast) =0$. By Lemma \ref{lemdet} we know that $\det(I-C_\ast D_\ast)=0$ for all $C\sim_f B$, or equivalently, for all $C\sim_f A$. This also means that $N(\overline{(c,C)}\circ \overline{f})=0$. By definition, $N_A(f)$ is a non-negative integer and hence, it follows by inequality (\ref{ineq}) that $$N_A(f)=0=\frac{1}{\#F}\sum_{B\sim_f A}|\det(I-B_\ast D_\ast)|.$$

Now, suppose that $\F=[(b,B)]_\Gamma$ is an essential fixed point class. The fact that$$N_A(f)\leq \sum_{\begin{subarray}{c}
\overline{(b,B)}\in \Gamma / \Lambda \\ B\sim_f A
\end{subarray}}N(\overline{(b,B)}\circ \overline{f})$$is not necessarily an equality comes from the fact that $\Psi$ is not injective. This is due to situations where$$[(b,B)]_\Gamma=[(c,C)]_\Gamma \text{, while } [(b,B)]_\Lambda\neq [(c,C)]_\Lambda.$$So, in order to find $N_A(f)$, we need to find the number of elements in $\mathcal{R}_\Lambda(f)$ which are mapped to the same element of $\mathcal{R}(f)$ by $\Psi$. First, we will show that this number has $\left|\Gamma/ \Lambda\right|$ as an upper bound. Suppose that $\overline{\gamma}_1=\overline{\gamma}_2 \in \Gamma/ \Lambda$, then $\gamma_2=\lambda\circ \gamma_1$, for $\lambda\in \Lambda$. If $$(c_1,C_1)=\gamma_1 \circ (b,B) \circ f_*(\gamma_1^{-1}) \textrm{ and } (c_2,C_2)=\gamma_2 \circ (b,B) \circ f_*(\gamma_2^{-1}),$$then an easy computation shows that $$(c_2,C_2)=\lambda\circ(c_1,C_1)\circ f_*(\lambda^{-1}),$$which means that  $[(c_1,C_1)]_\Lambda= [(c_2,C_2)]_\Lambda$.

\medskip

Now we will show that this upper bound is always attained by showing that $[(c_1,C_1)]_\Lambda= [(c_2,C_2)]_\Lambda$ implies that $\overline{\gamma}_1=\overline{\gamma}_2$ in $\Gamma / \Lambda$. Let $(d,D)$ be an affine homotopy lift of $f$. Suppose there exist a $\lambda \in \Lambda$, such that $$(c_1,C_1)=\gamma_1\circ (b,B)\circ  f_*(\gamma_1^{-1})=\lambda\circ (\gamma_2\circ (b,B)\circ  f_*(\gamma_2^{-1}))\circ f_*(\lambda^{-1})=\lambda\circ (c_2,C_2)\circ f_*(\lambda^{-1}).$$As an easy consequence, $$(b,B)=(\gamma_1^{-1}\circ \lambda\circ \gamma_2)\circ (b,B)\circ f_*(\gamma_1^{-1}\circ \lambda\circ \gamma_2)^{-1}.$$Note that $p(\Fix((b,B)\circ \tilde{f}))$ is an essential fixed point class and therefore $p(\Fix((b,B)\circ (d,D)))$ will also be an essential fixed point class. Hence, there exists an $x\in G$ such that $(b,B)\circ (d,D)(x)=x$. This also means that $$(\gamma_1^{-1}\circ \lambda\circ \gamma_2)\circ (b,B)\circ f_*(\gamma_1^{-1}\circ \lambda\circ \gamma_2)^{-1}\circ (d,D) (x)=x,$$which implies that $(\gamma_1^{-1}\circ \lambda\circ \gamma_2)^{-1}\cdot x$ is also in $p(\Fix((b,B)\circ (d,D)))$. By Proposition \ref{propaff}, we know that such a fixed point class is a singleton and hence, $(\gamma_1^{-1}\circ \lambda\circ \gamma_2)^{-1}\cdot x=x$. By the free action of $\Gamma$ on $G$, this implies that $\lambda\circ \gamma_2=\gamma_1$.

\medskip

Now, we know that $\Psi$ maps $\left|\Gamma/ \Lambda\right|$ different elements of $\mathcal{R}_\Lambda(f)$ to the element $[(b,B)]_\Gamma$. As $\F$ was chosen arbitrarily, we know that this holds for every essential fixed point class $[(c,C)]_\Gamma$ in $\mathcal{R}(f)$, for which $C\sim_f A$. So, this means that $$N_A(f)=\frac{1}{[\Gamma:\Lambda]}\sum_{\begin{subarray}{c}
\overline{(b,B)}\in \Gamma / \Lambda \\ B\sim_f A
\end{subarray}}N(\overline{(b,B)}\circ \overline{f}).$$In a similar way as in the proof of Theorem 3.4 in \cite{ll06-1}, we can now derive that$$N_A(f)=\frac{1}{\#F}\sum_{B\sim_f A}|\det(I-B_\ast D_\ast)|.$$
\end{proof}

\begin{remark}\label{remark essential of inessential}
During the proof of this theorem, we actually also proved that the fixed point class $p(\Fix((a,A)\circ \tilde{f})) $ is essential if and only if $\det(I-A_\ast D_\ast)\neq 0$. This is due to the fact that $p(\Fix((a,A)\circ \tilde{f}))$ can be lifted to a fixed point class $p'(\Fix((a,A)\circ \tilde{f}))$ with the same index. As this is a fixed point class for a map on a nilmanifold, we can use the result from \cite{anos85-1} or \cite{fh86-1}, which tells us that $p'(\Fix((a,A)\circ \tilde{f})) $ is essential if and only if $\det(I-A_\ast D_\ast)\neq 0$.
\end{remark}

It might be noteworthy to mention that a fixed point class $p(\Fix((b,B)\circ \tilde{f}))$ can be written as $p(\Fix((a,A)\circ \tilde{f}))$ if and only if $A\sim_f B$. So, in a certain sense, it is justified to say that $N_A(f)$ is the number of essential fixed point classes \textbf{above the $\sim_f$-equivalence class of $A$}. We will denote this equivalence class by $[A]$. Also note that every fixed point class above $A$ is simultaneously essential or inessential. Hence, it makes sense to talk about the (in)essential $\sim_f$-equivalence class $[A]$. We can also generalize these notions when considering $\sim_{f^{k}}$-equivalence classes of $A$. For these classes, we will use the notation $[A]_k$.

\medskip

Some easy corollaries of Theorem \ref{thmNAf} are the following.

\begin{corollary}
$$N_A(f)=\frac{\#\{B\in F\|B\sim_{f} A\}}{\# F}\cdot |\det(I-A_\ast D_\ast)|.$$
\end{corollary}
\begin{proof}
This follows easily by combining Lemma~\ref{lemdet} and Theorem~\ref{thmNAf}.
\end{proof}

\begin{corollary}
If all elements of $F$ are in the same $\sim_f$-equivalence class, then $$N(f)=N_A(f)=|\det(I-A_\ast D_\ast)|=|\det(I- D_\ast)|=|L(f)|.$$
\end{corollary}

This condition is for example satisfied when $f_*:\Gamma\to \Gamma$ maps every element into $\Gamma\cap G$. If this is the case, then $f_*$ induces the trivial morphism on the holonomy group, from which it follows easily that all elements in $F$ are in the same $\sim_f$-equivalence class.

\begin{corollary}
If every $\sim_{f}$-equivalence class in $F$ consists of a single element, then for all $A\in F$, $\det(I-A_\ast D_\ast)$ will be divisible by $\#F$.
\end{corollary}

This is for example the case when $f_*$ induces the identity morphism on $F$, while $F$ itself is an abelian group. For instance, in Example \ref{example Z6}, in section \ref{examples}.

\subsection{$\sim_f$-equivalence classes on different levels}

\begin{definition}\label{definition boosting equivalence classes}
Let $A\in F$ be an element of the holonomy group $F$ of $\Gamma\backslash G$. Let $k|n$. Then we define $\gamma_{nk}(A)$ to be the following subset of $F$:$$\{C\in F \| \textrm{ there exists } (a,A)\in \Gamma, \textrm{such that } (c,C)=\gamma_{nk}(a,A)=(a,A)f_*^k(a,A)\dots f_*^{n-k}(a,A)\}.$$
\end{definition}

With this definition in mind, we can prove the following lemma.

\begin{lemma}\label{lemma boosting funct on f-conjugacy classes}
If $A\sim_{f^k} B$, then for any $ C \in \gamma_{nk}(A)$, there exists a $D\in \gamma_{nk}(B)$ such that $C \sim_{f^n} D$.
\end{lemma}
\begin{proof}
By Lemma \ref{lemma alternatief definitie}, for any $(a,A) \in \Gamma$, there exist $(b,B), \gamma \in \Gamma$, such that $$\gamma\circ (a,A)\circ f_*^k(\gamma^{-1})=(b,B).$$Because we picked $(a,A)$ arbitrarily, $C$, with $(c,C)=\gamma_{nk}(a,A)$, is also chosen arbitrarily in the set $\gamma_{nk}(A)$. Now, the following $(d,D)$ fulfills the necessary conditions:$$(d,D)=(b,B)f_*^k(b,B)\dots f_*^{n-k}(b,B).$$Indeed, by using the relation $\gamma\circ (a,A)\circ f_*^k(\gamma^{-1})=(b,B)$, we see$$(d,D)=(\gamma\circ (a,A)\circ f_*^k(\gamma^{-1}))f_*^k(\gamma\circ (a,A)\circ f_*^k(\gamma^{-1}))\dots f_*^{n-k}(\gamma\circ (a,A)\circ f_*^k(\gamma^{-1})).$$Since $f_*$ is a morphism, a simple computation shows that $$(d,D)=\gamma\circ (c,C) \circ f_*^n(\gamma^{-1}).$$
\end{proof}

With Definition \ref{definition boosting equivalence classes}, we actually try to define boosting functions in terms of $\sim_f$-equivalence classes. This might not necessarily be well-defined, in the sense that it might happen that for $A \sim_{f^k} B$, not every element in $\gamma_{nk}(A)$ is automatically $\sim_{f^n}$-conjugated with every element in $\gamma_{nk}(B)$. Note that Lemma \ref{lemma boosting funct on f-conjugacy classes} tells us that this would be the case if every element in $\gamma_{nk}(A)$ is in the same $\sim_{f^n}$-equivalence class. Because of Corollary \ref{corollary D invertible}, we know that $\gamma_{nk}(A)$ will be a singleton whenever $D$ is invertible, so in that case, these boosting functions are well-defined on the equivalence classes. Note that it might also not necessarily be true that a  $\gamma_{nk}([A]_k)$, by which we mean the set $$\bigcup_{B\sim_{f^{k}}A}\gamma_{nk}(B),$$is a full $\sim_{f^n}$-equivalence class.
\medskip

Although boosting functions might generally not behave well on $\sim_{f^k}$-equivalence classes, we can still use them as a tool for the computation of $NF_n(f)$. The reason for this is the fact that two fixed point classes above the same equivalence class $[A]$ will boost in the exact same way. So, if we know how one fixed point class behaves, every fixed point class above the same equivalence class will behave in the same way.

\medskip

The following lemma is actually all we need to prove the statement above.
\begin{lemma}\label{lemma determinants for boosted}
Let $(d,D)$ be an affine homotopy lift of $f$. If $B, C\in \gamma_{nk}(A)$, then $$\det(I-B_\ast D_\ast^n)=\det(I-C_\ast D_\ast^n).$$
\end{lemma}
\begin{proof}
Since there exist $(a_1,A)$ and $(a_2,A)$, such that $\gamma_{nk}(a_1,A)=(b,B)$ and $\gamma_{nk}(a_2,A)=(c,C)$, for certain $b,c\in G$, we know that $$((a_1,A)(d,D)^k)^{\frac{n}{k}}=(b,B)(d,D)^n \textrm{ and }((a_2,A)(d,D)^k)^{\frac{n}{k}}=(c,C)(d,D)^n.$$So, by just looking at the rotational part, we see $$BD^n=(AD^k)^{\frac{n}{k}}=CD^n.$$By taking the differential, we obtain the desired result.
\end{proof}

\begin{proposition}\label{prop similar boosting}
Suppose $p(\Fix((a,A)\circ \tilde{f}^k))$ is a fixed point class at level $k$ of a continuous map $f$ on an infra-nilmanifold. If this fixed point class boosts (in)essentially to level $n$, then every fixed point class of the form $p(\Fix((b,B)\circ \tilde{f}^k))$, with $B\in [A]_k$ also boosts (in)essentially to level $n$.
\end{proposition}
\begin{proof}
Suppose that $(d,D)$ is an affine homotopy lift of $f$. By Remark \ref{remark essential of inessential}, we know that $p(\Fix((a,A)\circ \tilde{f}^k))$ is an essential fixed point class if and only if $\det(I-A_\ast D_\ast^k)\neq 0$. Take an arbitrary fixed point class $p(\Fix((b,B)\circ \tilde{f}^k))$, with $B\in [A]_k$. Because of Lemma \ref{lemdet}, we already know that $\det(I-B_\ast D_\ast^k)\neq 0$ and that $p(\Fix((b,B)\circ \tilde{f}^k))$ is an essential fixed point class.

\medskip

Now suppose that $\gamma_{nk}(a,A)=(c,C)$ and $\gamma_{nk}(b,B)=(e,E)$. This means that $C\in \gamma_{nk}(A)$ and $E\in \gamma_{nk}(B)$. By Lemma \ref{lemma boosting funct on f-conjugacy classes}, we know that there exists $E_0\in \gamma_{nk}(B)$, such that $C\in [E_0]_n$. Lemma \ref{lemdet} now tells us that $$\det(I-C_\ast D_\ast^n)=\det(I-E_{0 \ast} D_\ast^n).$$By Lemma \ref{lemma determinants for boosted} and by the fact that $E,E_0\in \gamma_{nk}(B)$, we also know that $$\det(I-E_\ast D_\ast^n)=\det(I-E_{0 \ast} D_\ast^n).$$Because $\det(I-C_\ast D_\ast^n)=\det(I-E_\ast D_\ast^n)$, we know that $[(a,A)]_k$ boosts essentially to level $n$ if and only if $[(b,B)]_k$ boosts essentially to level $n$. A similar thing applies for inessential boosting.
\end{proof}

\subsection{Examples}\label{examples}

So, we know that every fixed point class above $[A]$ boosts in exactly the same way and we also know that the number of essential fixed point classes above $[A]$ equals $N_A(f)$. This is a tool we can use to compute $\# IIB_n(f)$ in a more efficient way. We show how this can be done by looking at a few examples.

\begin{example}
Let us first try to compute the Nielsen periodic numbers of the maps in Example~\ref{Example not weakly Jiang}. We will use the matrix description from \cite{duga14-1}.
Let the Klein bottle group be generated by the following two affine transformations:
$$\alpha=(a,A)=\left(\begin{pmatrix}
\frac{1}{2}\\
\frac{1}{2}\\
\end{pmatrix}, \begin{pmatrix}
1&0\\
0&-1
\end{pmatrix}\right) \textrm{ and } \beta=(e_2, \Id),$$where $e_2$ denotes the second element of the standard basis of $\R^2$. Suppose that $k\neq 1$ is odd and that $p\in \R$. Then, the map induced by $$(d,D)=\left(\begin{pmatrix}
p\\
\frac{1}{2}\\
\end{pmatrix}, \begin{pmatrix}
k&0\\
0&-1
\end{pmatrix}\right)$$will induce the same morphism described in Example \ref{Example not weakly Jiang}.

\medskip

It is clear that $D$ commutes with both $A$ and $\Id$, so $f_*$ induces the identity morphism on $F$ and therefore, every $\sim_{f^n}$-equivalence class consists of precisely one element. Also, it easy to compute that $[A]_l$ is essential if and only if $l$ is even, while $[\Id]_l$ is essential if and only if $l$ is odd. Now, suppose that $m=ql$, then a simple computation shows that $[\Id]_l$ always boosts to $[\Id]_m$. Also, $[A]_l$ will boost to $[A]_m$ if $q$ is odd and to $[\Id]_m$ if $q$ is even. The reason for this, lies in the following computation:$$(AD^l)^q=A^qD^{ql}=A^qD^m.$$All together, we see that every even boost ($q$ is even) of an essential fixed point class is inessential while every odd boost is essential.

\medskip

As a consequence, we see that if $n$ is odd, $IIB_n(f)$ is the empty set and by Proposition \ref{propIIB}, it follows that $$NF_n(f)=N(f^n).$$On the other hand, if $n$ is even, the only essential fixed point classes that boost inessentially to level $n$ pass through level $\frac{n}{2}$. Every essential fixed point class at this level will boost inessentially to level $n$ and every element in $IIB_{\frac{n}{2}}(f)$ will also boost inessentially to level $n$. Therefore: $$NF_n(f)=N(f^n)+NF_{\frac{n}{2}}(f).$$

For the sake of completeness, the case where $k=1$ also gives us a map on the Klein bottle. In this case, an easy computation shows that $N(f^n)=0$ for every $n$. As a consequence, every fixed point class at every level is inessential and hence $NF_n(f)=0$, for every integer $n$.
\end{example}

The following example will illustrate the use of $\sim_f$-equivalence classes a little more. 

\begin{example}\label{example Z3}
Let $\Gamma$ be the Bieberbach group with generators:$$(a,A)=\left(\begin{pmatrix}
0\\
0\\
\frac{1}{3}\\
\end{pmatrix}, \begin{pmatrix}
-1&1&0\\
 -1 & 0 & 0\\
 0&0&1 \\
\end{pmatrix}\right)
\textrm{, }(e_1, \Id)\textrm{ and }(e_2, \Id).$$In \cite{duga14-1} one can find that the affine map $$(d,D)=\left(\begin{pmatrix}
0\\
0\\
0\\
\end{pmatrix}, \begin{pmatrix}
0&1&0\\
 1 & 0 & 0\\
 0&0&2 \\
\end{pmatrix}\right)$$induces a continuous map on the flat manifold $\Gamma\backslash\R^3 $.

\medskip

An easy computation shows that $DA=A^2D$ and that $f_*$ induces a morphism $\overline{f}_*$ on the holonomy group $\Z_3$, such that $\overline{f}_*^2=\Id$. So, whenever $k$ is even, every $\sim_{f^k}$-equivalence class is a singleton.
\medskip

To determine $[\Id]_k$, with $k$ odd, note that $\overline{f}_*^k=\overline{f}_*$. Also, $\overline{f}_*(A)=A^2$ and $\overline{f}_*(A^2)=A$. So, the following are certainly subsets of $[\Id]_k$:
$$A\Id f_\#(A)^{-1}=\{A^2\} \textrm{ and } A^2\Id f_\#(A^2)^{-1}=\{A\}.$$Hence, $[\Id]_k=F$ for all odd $k$. 

\medskip

An easy computation shows that $[\Id]_k$, with $k$ odd, is always inessential. As a consequence, for every odd $n$, $$NF_n(f)=N(f^n)=0.$$When $k$ is even, $[\Id]_k$ is inessential, while $[A]_k$ and $[A^2]_k$ are essential. In this case, every element of $F$ commutes with $D^k$, as $\overline{f}_*^2$ is the identity morphism. Hence, the class $[A^i]_k$ boosts to the class $[A^{ip}]_{pk}$. So, an essential class can only boost inessentially if $p\equiv 0 \mod 3$. 

\medskip

In Figure \ref{fig1}, a scheme can be found where all these boosting relations are shown up to level $6$. In this scheme, inessential and essential fixed point classes are denoted by a circle and a square respectively. Only the boosting from an essential to an inessential class are drawn, since these are the only ones that need to be considered for the computation of $NF_n(f)$.

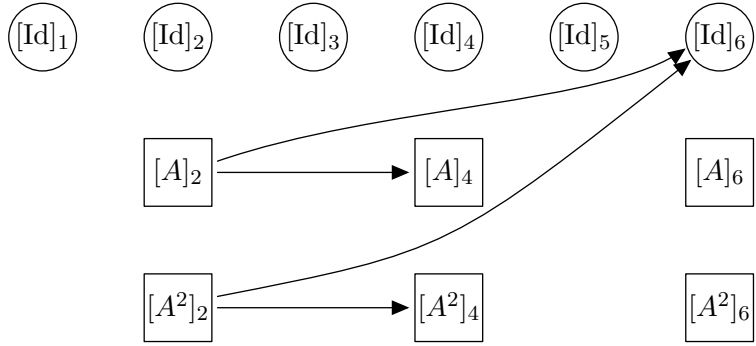
\begin{figure}

\centering
\begin{tikzpicture}[line cap=round,line width = .5pt,line join=round, >=triangle 45, x=1.5cm,y=1.5cm]

%cirkels
\draw[color=black](0,1.2) circle (0.3);
\draw[color=black](1.2,1.2) circle (0.3);
\draw[color=black](.9,-.9) rectangle (1.5,-1.5);
\draw[color=black](.9,-.3) rectangle (1.5,.3);
\draw[color=black](2.4,1.2) circle (0.3);
\draw[color=black](3.6,1.2) circle (0.3);
\draw[color=black](3.3,-.9) rectangle (3.9,-1.5);
\draw[color=black](3.3,-.3) rectangle (3.9,.3);
\draw[color=black](4.8,1.2) circle (0.3);
\draw[color=black](6,1.2) circle (0.3);
\draw[color=black](5.7,-.9) rectangle (6.3,-1.5);
\draw[color=black](5.7,-.3) rectangle (6.3,.3);

%tekst
\draw[color=black] (0,1.2) node {$[\Id]_1$};
\draw[color=black] (1.2,1.2) node {$[\Id]_2$};
\draw[color=black] (1.2,0) node {$[A]_2$};
\draw[color=black] (1.2,-1.2) node {$[A^2]_2$};
\draw[color=black] (2.4,1.2) node {$[\Id]_3$};
\draw[color=black] (3.6,1.2) node {$[\Id]_4$};
\draw[color=black] (3.6,0) node {$[A]_4$};
\draw[color=black] (3.6,-1.2) node {$[A^2]_4$};
\draw[color=black] (4.8,1.2) node {$[\Id]_5$};
\draw[color=black] (6,1.2) node {$[\Id]_6$};
\draw[color=black] (6,0) node {$[A]_6$};
\draw[color=black] (6,-1.2) node {$[A^2]_6$};

%lijnen
\draw[->,color=black] (1.55,0) -- (3.25,0);
\draw[->,color=black] (1.55,-1.2) -- (3.25,-1.2);
\draw[->,color=black] (1.55,.1) .. controls (3,.6) and (5,.6) .. (5.7,1.1);
\draw[->,color=black] (1.55,-1.1) .. controls (3.6,-.7) .. (5.75,1);
\end{tikzpicture}

\caption{A scheme of $\sim_{f^k}$-equivalence classes at different levels for Example \ref{example Z3}} \label{fig1}
\end{figure}

\medskip

So, suppose that $n=3^pq$ is even, such that $\gcd(3,q)=1$, then $$NF_n(f)=\sum_{i=0}^{p} N(f^{3^iq}).$$
\end{example}

By the following example, we clarify the use of Theorem \ref{thmNAf}.

\begin{example}\label{example Z6}
Let $\Gamma$ be the Bieberbach group with generators:
$$(a,A)=\left(\begin{pmatrix}
0\\
0\\
\frac{1}{6}
\end{pmatrix}, \begin{pmatrix}
1 & -1 &0\\
1& 0 &0 \\
0&0 & 1
\end{pmatrix}\right), (e_1,\Id)\textrm{ and }  (e_2,\Id).$$

In \cite{duga14-1}, one can find that the following affine map induces a map on the infra-nilmanifold $\Gamma\backslash \R^3$:

$$(d,D)=\left(\begin{pmatrix}
0\\
0\\
0
\end{pmatrix}, \begin{pmatrix}
0 & 1 &0\\
-1& 1 &0 \\
0&0 &7
\end{pmatrix}\right).$$

Note that every element of $F$ commutes with $D$, so every $\sim_f$-equivalence class consists of precisely one element. This also means that the class $[A^i]_k$ boosts to the class $[A^{ip}]_{pk}$. It is also quite easy to compute that $[A^p]_k$ is inessential if and only if $p\equiv k \mod 6$. This boosting scheme can be found in Figure \ref{fig2}.

\begin{figure}

\centering
\begin{tikzpicture}[line cap=round,line width = .5pt,line join=round, >=triangle 45, x=1.8cm,y=1.5cm]

%cirkels
\draw[color=black](0,2) circle (0.3);
\draw[color=black](-.3,.9) rectangle (.3,1.5);
\draw[color=black](-.3,.1) rectangle (.3,.7);
\draw[color=black](-.3,-.7) rectangle (.3,-.1);
\draw[color=black](-.3,-1.5) rectangle (.3,-.9);
\draw[color=black](-.3,-2.3) rectangle (.3,-1.7);

\draw[color=black](2.1,1.7) rectangle (2.7,2.3);
\draw[color=black](2.4,1.2) circle (0.3);
\draw[color=black](2.1,.1) rectangle (2.7,.7);
\draw[color=black](2.1,-.7) rectangle (2.7,-.1);
\draw[color=black](2.1,-1.5) rectangle (2.7,-.9);
\draw[color=black](2.1,-2.3) rectangle (2.7,-1.7);

\draw[color=black](4.5,1.7) rectangle (5.1,2.3);
\draw[color=black](4.5,.9) rectangle (5.1,1.5);
\draw[color=black](4.8,.4) circle (0.3);
\draw[color=black](4.5,-.7) rectangle (5.1,-.1);
\draw[color=black](4.5,-1.5) rectangle (5.1,-.9);
\draw[color=black](4.5,-2.3) rectangle (5.1,-1.7);
%tekst
\draw[color=black] (0,2) node {$[A]_1$};
\draw[color=black] (0,1.2) node {$[A^2]_1$};
\draw[color=black] (0,.4) node {$[A^3]_1$};
\draw[color=black] (0,-.4) node {$[A^4]_1$};
\draw[color=black] (0,-1.2) node {$[A^5]_1$};
\draw[color=black] (0,-2) node {$[\Id]_1$};

\draw[color=black] (2.4,2) node {$[A]_2$};
\draw[color=black] (2.4,1.2) node {$[A^2]_2$};
\draw[color=black] (2.4,.4) node {$[A^3]_2$};
\draw[color=black] (2.4,-.4) node {$[A^4]_2$};
\draw[color=black] (2.4,-1.2) node {$[A^5]_2$};
\draw[color=black] (2.4,-2) node {$[\Id]_2$};

\draw[color=black] (4.8,2) node {$[A]_3$};
\draw[color=black] (4.8,1.2) node {$[A^2]_3$};
\draw[color=black] (4.8,.4) node {$[A^3]_3$};
\draw[color=black] (4.8,-.4) node {$[A^4]_3$};
\draw[color=black] (4.8,-1.2) node {$[A^5]_3$};
\draw[color=black] (4.8,-2) node {$[\Id]_3$};

%lijnen
\draw[->,color=black] (.35,1.15) -- (2.05,-.3);
\draw[->,color=black] (.35,.35) -- (2.05,-1.9);
\draw[->,color=black] (.35,-.35) -- (2.05,1.2);
\draw[->,color=black] (.35,-1.1) -- (2.05,-.4);
\draw[->,color=black] (.35,-1.95) -- (2.05,-2);

\draw[->,color=black] (.35,1.2) .. controls (3,2) .. (4.45,-1.9);
\draw[->,color=black] (.35,.4) .. controls (3,1) .. (4.45,.5);
\draw[->,color=black] (.35,-.4) .. controls (2,-2.5) and (3,-1) .. (4.45,-2);
\draw[->,color=black] (.35,-1.2) .. controls (2,-.1) and (3.7,-1.8) .. (4.5,.2);
\draw[->,color=black] (.35,-2.05) .. controls (2.4,-2.5) .. (4.45,-2.05);

\end{tikzpicture}

\caption{A scheme of $\sim_{f^k}$-equivalence classes at the lowest levels for Example \ref{example Z6}} \label{fig2}
\end{figure}
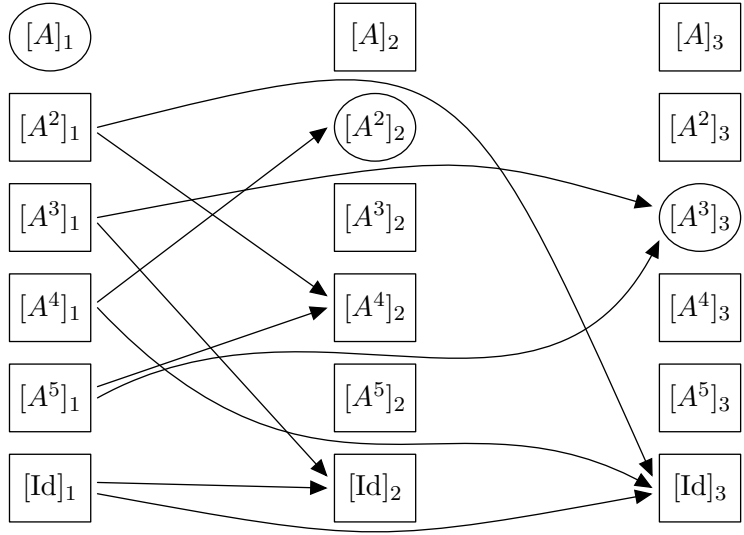

\medskip

The only classes at level $1$ that boost inessentially to level $2$ (to $[A^2]_2$), are the inessential class $[A]_1$ and the essential class $[A^4]_1$. It is therefore clear that $\# IIB_2(f)=N_{A^4}(f)$. By Theorem \ref{thmNAf} and Proposition \ref{propIIB}:$$NF_2(f)=N(f^2)+\frac{|\det(I-A^4D)|}{6}.$$
In a similar way, one can see that the only classes that boost to $[A^3]_3$ are the inessential class $[A]_1$ and the essential classes $[A^3]_1$ and $[A^5]_1$. Hence,$$NF_3(f)=N(f^3)+ \frac{|\det(I-A^3D)|}{6}+ \frac{|\det(I-A^5D)|}{6}.$$
Computing $NF_4(f)$ becomes a little more tricky, since fixed point classes at both level $1$ and $2$ can boost to inessential fixed point classes of level $4$. With an easy computation, we see that the classes that boost to $[A^4]_4$ are $[A]_1,[A^4]_1,[A^2]_2$ and $[A^5]_2$. Note that we already knew that $[A]_1$ is inessential and that the essential class $[A^4]_1$ boosts inessentially to $[A^2]_2$. Therefore, there are no essential classes at level $1$ that boosts essentially to level $2$ and inessentially to level $4$. This means that $$NF_4(f)=N(f^4)+ \frac{|\det(I-A^5D^2)|}{6}+\frac{|\det(I-A^4D)|}{6}.$$
As $[A^i]_k$ boosts to $[A^{ip}]_{pk}$, we know that this boosting relation is a bijection between the $\sim_{f^k}$-equivalence classes and the $\sim_{f^{pk}}$-equivalence classes if and only if $p$ is invertible modulo $6$. Now, suppose $n>0$ is an integer, such that $\gcd(n,6)=1$. Note that every divisor of $n$ will also be relatively prime to $6$. Because there is only one inessential class at each level and because there is a bijection between the classes at different levels and because maps on infra-nilmanifolds are essentially reducible, every essential class that boosts to level $n$ will do so in an essential way. Therefore, if $\gcd(n,6)=1$, $$NF_n(f)=N(f^n).$$Whenever $n$ has many prime factors $2$ and $3$, it will be much harder to compute $NF_n(f)$, because many inessential boosts occur and we have to keep track of all these boostings in order to not count some of them multiple times. As an example, let us compute $NF_6(f)$. Note that $[\Id]_6$ is the only inessential class at level $6$. The classes that boost to $[\Id]_6$ are $[\Id]_3, [A^3]_3, [\Id]_2, [A^2]_2,[A^4]_2$ and all classes at level $1$. The only essential classes at level $1$ that boost to essential classes at both level $2$ and level $3$, are $[\Id]_1$ and $[A^2]_1$. Also, there are no essential classes at level $1$ that boost to inessential classes at both level $2$ and level $3$. Hence,
\begin{equation}\nonumber
\begin{split}
NF_6(f)=N(f^6)+\frac{|\det(I-D^3)|}{6}+\frac{|\det(I-D^2)|}{6}+\frac{|\det(I-A^4D^2)|}{6}\\
-\frac{|\det(I-D)|}{6}-\frac{|\det(I-A^2D)|}{6}.
\end{split}
\end{equation}
Here, these last two terms are precisely the number of essential fixed point classes at level $1$ that boost essentially to level $2$ and level $3$. As they are counted double, we have to subtract them once.
\end{example}

As one can see from this last example, it can be very hard to compute $NF_n(f)$. The tools in this section are useful, but they still require a lot of manual labor. Looking at these examples, it is not unthinkable that there might not exist a general formula for $NF_n(f)$.

\section{Some properties of affine maps on infra-nilmanifolds}
In the last section of this paper, we will look specifically at affine maps on infra-nilmanifolds in order to derive a nice property of these maps (Theorem \ref{theorem uiteindelijk boost inessentieel}).

\begin{lemma}\label{lemma alternatief lemma bram 3.1}
If $A,B \in \GL_n(\C)$ and $D\in \C^{n\times n}$, such that $DA=BD$, then, for all $n>0$ it holds that $$\det(I-(AD)^n)=\det(I-(BD)^n).$$
\end{lemma}
\begin{proof}
Using the multiplicative properties of the determinant, we find the following equalities:$$\det(I-(AD)^n)=\det(A^{-1})\det(I-(AD)^n)\det(A)=\det(I-(DA)^n)=\det(I-(BD)^n).$$\end{proof}

Remember that a continuous map $f$ will be called \textbf{Wecken} if and only if $\#\Fix(f)=N(f)$.  

\begin{theorem}\label{theorem uiteindelijk boost inessentieel}
Let $\overline{(d,D)}$ be an affine map on an infra-nilmanifold. Suppose that there exists at least one $k$, for which $N(\overline{(d,D)}^k)\neq 0$. Then, $D_\ast$ is semi-hyperbolic if and only if $\overline{(d,D)^k}$ is a Wecken map for every $k$.   
\end{theorem}
\begin{proof}
First, suppose that $D_\ast$ is semi-hyperbolic. Just like in the proof of Theorem \ref{thmNF=Nf)}, we know that every fixed point class at every level is essential. By Proposition \ref{propaff}, it follows that $\overline{(d,D)^k}$ is a Wecken map.

\medskip

On the other hand, suppose that $D_\ast$ is not semi-hyperbolic. This means that there exists an eigenvalue $\lambda$ of $D_\ast$, such that $\lambda^d=1$. Now we will show that every essential fixed point class will eventually be boosted to an inessential fixed point class. Pick an essential fixed point class $[(a,A)]_k$. This is possible, because not all Nielsen numbers are $0$. Let $l$ be an arbitrary positive integer and set $m=kl$. Consider the fixed point class $\gamma_{km}([(a,A)])$. This coincides with the set $p(\Fix(((a,A)\circ(d,D)^k)^l))$. Now suppose that $$(b,B)\circ (d,D)^{m}=((a,A)\circ(d,D)^k)^l.$$This means that $BD^{m}=(AD^k)^l$ and by Remark~\ref{remark essential of inessential}, we now know that $\gamma_{km}([(a,A)])$ is inessential if and only if $$\det(I-(A_\ast D_\ast^k)^l)=\det(I-B_\ast D_\ast^{m})=0.$$By combining Lemma~\ref{lemma alternatief lemma bram 3.1} and Lemma \ref{Lemma Bram}, we see there exists a $C\in F$ and a positive integer $p$, such that $$\det(I-(A_\ast D_\ast^k)^l)=\det(I-(C_\ast D_\ast^k)^l)\textrm{ and } (C_\ast D_\ast^k)^p=D_\ast^{kp}.$$By taking $l=\lcm(p,d)$, we know that $(C_\ast D_\ast^k)^l=D_\ast^{kl}$. Also, $1$ is an eigenvalue of $D_\ast^{kl}$. By combining all of the above, we see that $$\det(I-B_\ast D_\ast^{m})=\det(I-(A_\ast D_\ast^k)^l)=\det(I-(C_\ast D_\ast^k)^l)=\det(I-( D_\ast^k)^l)=0.$$As there is certainly one essential fixed point class $[(a,A)]_k$, we know that it will boost to an inessential fixed point class $[(b,B)]_m$. This actually means that $[(a,A)]_k\subset [(b,B)]_m$, which implies that the inessential fixed point class $[(b,B)]_m$ is non-empty, which implies that $\overline{(d,D)}^m$ is not a Wecken map.
\end{proof}

\begin{corollary}\label{corWecken}
Suppose that there exists at least one $k$, such that $N(\overline{(d,D)}^k)\neq 0$. Whenever $\Fix(\overline{(d,D)}^k)$ is finite for every $k$, $\overline{(d,D)}$ will be Wecken at every level and $D_\ast$ is semi-hyperbolic.
\end{corollary}
\begin{proof}
Due to Proposition~\ref{propaff}, we know that every non-empty inessential fixed point class contains infinitely many fixed points.
\end{proof}

\begin{corollary}\label{cor NF_n(f)>Nfn}
Suppose that $f$ is a continuous map on an infra-nilmanifold that is not semi-hyperbolic. Suppose that $N(f^k)\neq 0$ for at least one $k$. Then, at certain levels, there exist non-empty inessential fixed point classes. Also, there exist $n>0$, such that $$NF_n(f)>N(f^n).$$
\end{corollary}
\begin{proof}
By examining the proof of Theorem~\ref{theorem uiteindelijk boost inessentieel}, we see that there will exist an essential fixed point class which boosts to an inessential fixed point class. Therefore, this inessential fixed point class will be non-empty. On the other hand, due to Proposition~\ref{propIIB}, the second statement follows.
\end{proof}

\begin{corollary}\label{cor NF=Nf}
Suppose that $f$ is a continuous map on an infra-nilmanifold. Then, $NF_n(f)=N(f^n)$, for all $n$ if and only if $f$ is a semi-hyperbolic map or $N(f^n)=0$, for all $n$.
\end{corollary}
\begin{proof}
When dealing with semi-hyperbolic maps, the statement follows from Theorem \ref{thmNF=Nf)} and Corollary \ref{cor NF_n(f)>Nfn}. Hence, the only thing left to prove is that $NF_n(f)=0$ if $N(f^n)=0$ for all $n$. This actually follows from Proposition \ref{propIIB}. As all fixed point classes at all levels are inessential, we know that $\#IIB_n(f)=0$. As we already knew that $N(f^n)=0$, it follows by Proposition \ref{propIIB} that $NF_n(f)=0$.
\end{proof}

Again, we can translate some of these results into comparable results concerning dynamical zeta functions.

\begin{corollary}
Suppose that $f$ is continuous map on an infra-nilmanifold. Then, $NF_f(z)=N_f(z)$ if and only if $f$ is a semi-hyperbolic map or $N(f^n)=0$, for all $n$.
\end{corollary}
\begin{proof}
This follows immediately from Corollary \ref{cor NF=Nf}.
\end{proof}

We can actually say something more about another zeta function. In \cite{fels88-1}, the following dynamical zeta function was defined: $$M_g(z)=\exp\left(\sum_{k=1}^\infty \frac{\#\Fix(g^k)z^k}{k}\right).$$When working with affine semi-hyperbolic maps on infra-nilmanifolds, we actually know how this zeta function looks, due to the following result.

\begin{corollary}
Suppose that $g$ is an affine semi-hyperbolic map on an infra-nilmanifold, then $$M_g(z)=NF_g(z).$$ 
\end{corollary}
\begin{proof}
Every such a map is a Wecken map on every level, due to Theorem \ref{theorem uiteindelijk boost inessentieel}. From this it follows, for all $k>0$, that $$\#\Fix(g^k)=N(g^k)=NF_k(g).$$ 
\end{proof}

This result partially answers a question asked in \cite{jezi03-01} in the case of infra-nilmanifolds. Given a map $f$ on a manifold, the author of this paper asked if it would be possible to find a map $g$ homotopic to $f$, such that $$\#\Fix(g^k)=NF_k(g)=NF_k(f),$$for all $k$. This question is equivalent to asking whether there exists a map $g$, homotopic to $f$, such that $M_g(z)$ and $NF_g(z)$ coincide.

\section*{Acknowledgements}

I would like to thank my advisor Karel Dekimpe for all useful comments on the first versions of this paper.

\medskip

This work was supported by the research fund KU Leuven.

\bibliography{G:/algebra/ref}
\bibliographystyle{G:/algebra/ref}

\end{document}